\documentclass[11pt]{amsart}
\usepackage{amsmath}
\usepackage[english,  activeacute]{babel}
\usepackage[latin1]{inputenc}
\usepackage{amssymb}
\usepackage{amsthm}
\usepackage{graphics,graphicx}
\usepackage{array}
\usepackage{a4wide}
\allowdisplaybreaks
\usepackage{color, url}
\usepackage{float}

\theoremstyle{plain}
\newtheorem{theorem}{Theorem}

\newtheorem{lemma}[theorem]{Lemma}
\theoremstyle{definition}

\date{\today}

\title[Enumeration of circular permutations]{Enumerating circular permutations avoiding the vincular pattern $\overline{23}$41}

\begin{document}\author[T. Mansour]{Toufik Mansour}
\address{Department of Mathematics, University of Haifa,
3498838 Haifa, Israel}
\email{tmansour@univ.haifa.ac.il}
\author[M. Shattuck]{Mark Shattuck}
\address{Department of Mathematics, University of Tennessee,
37996 Knoxville, TN}
\email{shattuck@math.utk.edu}

\begin{abstract} In this paper, we find an explicit formula for the generating function that counts the circular permutations of length $n$ avoiding the pattern $\overline{23}41$ whose enumeration was raised as an open problem by Rupert Li. This then completes in all cases the enumeration of circular permutations that avoid a single vincular pattern of length four containing one vinculum.  To establish our results, we introduce three auxiliary arrays which when taken together refine the cardinality of the class of permutations in question.  Rewriting the recurrences of these arrays in terms of generating functions leads to functional equations which are solved by various means including the kernel method and iteration.
\end{abstract}

\maketitle

\section{Introduction}

The problem of pattern avoidance has been an ongoing object of research in combinatorics over the past few decades.  We refer the reader to the text \cite{Kit} and references contained therein.  Considered originally on permutations \cite{Knuth,SS}, the problem has been studied on several other finite discrete structures including compositions, set partitions and words (see, e.g., the texts \cite{HeuM,Mans}).  Further, various extensions of the basic avoidance problem have been obtained by stipulating that an occurrence of a pattern must meet certain requirements.  In this paper, we consider the avoidance of a particular pattern of length four by circular permutations wherein the first two entries in an occurrence of the pattern satisfy an adjacency condition.

A \emph{linear} permutation of $[n]=\{1,\ldots,n\}$ is any sequential arrangement of the elements of $[n]$ written in a row.  In contrast, in a \emph{circular} permutation, the elements of $[n]$ are written in some order about the circumference of a circle instead of along a line.  Thus, one may regard the position of $1$ as being fixed and hence there are $(n-1)!$ circular permutations of $[n]$.  Circular permutations are encountered frequently outside of the realm of pattern avoidance; for an example of a recent application, see \cite{KW} where they are used in the study of the inhomogeneous totally asymmetric simple exclusion process on a ring.

Given linear permutations $\pi=\pi_1\cdots\pi_n$ and $\tau=\tau_1\cdots\tau_m$ where $1 \leq m \leq n$, then $\pi$ is said to \emph{contain} the pattern $\tau$ if there exist
$1 \leq i_1<\cdots<i_m \leq n$ such that $\pi_{i_1}\cdots \pi_{i_m}$ is order isomorphic to $\tau$.  Otherwise, $\pi$ is said to \emph{avoid} $\tau$.  This definition is extended to circular permutations by allowing for subsequences starting near the end of $\pi$ to wrap back around to the beginning.  More precisely, given $\pi=\pi_1\pi_2\cdots\pi_n$, let $$S=\{\pi_1\pi_2\cdots\pi_n,\pi_n\pi_1\cdots\pi_{n-1},\pi_{n-1}\pi_n\pi_1\cdots\pi_{n-2},\ldots,\pi_2\pi_3\cdots \pi_n\pi_1\}$$ consist of the permutations obtained from $\pi$ by applying repeatedly the cyclic shift operation where the last letter is moved to the front.  Then the circular permutation $\pi$ contains the pattern $\tau$ if and only if some member of $S$ contains $\pi$ per the definition given above in the linear case. If this fails to occur, then $\pi$ avoids $\tau$ in the circular sense.  Geometrically, an occurrence of $\tau$ in a circular permutation $\pi$ happens when one starts at any fixed position $x$ along the circle and, proceeding clockwise from $x$, encounters a subsequence of $\pi$ isomorphic to $\tau$ prior to returning to $x$.  Note that all patterns that can be obtained from $\tau$ by cyclic rotation are equivalent to $\tau$ in circular pattern avoidance.  This is in addition to the usual symmetries such as reversal and complementation which also apply to linear patterns.

The study of pattern avoidance in circular permutations was initiated by Callan \cite{Cal} who considered the case of a single classical pattern of length four for which there are three distinct Wllf equivalence classes (with representative members $1234$, $1324$ and $1342$).  Simple explicit formulas were found in \cite{Cal} for the number of circular permutations of $[n]$ avoiding each of the three patterns.  These enumerative results were later extended to all sets of patterns of length four by Domagalski et al.~\cite{DLM}.  See also the related paper by Vella \cite{Ve}.  Gray et al.~extended Callan's work by considering cyclic packing of patterns \cite{GLW1}
and patterns in colored circular permutations \cite{GLW2}. Further related generating function formulas which count circular permutations according to the number of occurrences of a subword pattern starting with 1 have been found by Elizalde and Sagan \cite{ES} in terms of the comparable formulas from the linear case.

Babson and Steingr\`{\i}msson \cite{BS} introduced the notion of \emph{vincular} pattern avoidance where certain adjacent elements of the pattern $\tau$ are required to be adjacent in an occurrence of $\tau$ within $\pi$.  The elements that must be adjacent are overlined in $\tau$ and each such pair of elements is referred to as a \emph{vinculum}.  For example, $\pi=41523$ contains two subsequences that are isomorphic to $\tau=231$, namely $452$ and $453$, but only $452$ is an occurrence of $2\overline{31}$ as the $5$ and $2$ are adjacent.  Vincular patterns in which every possible pair of adjacent elements is overlined are known as \emph{subwords}, whereas those in which no pair is overlined correspond to the \emph{classical} patterns.  The notion of vincular pattern avoidance in circular permutations is defined in analogy to the classical case described above.

Note that there are $4!=24$ vincular patterns of length four containing one vinculum in the circular case, as it may be assumed that the first two letters in a pattern are those that are overlined.  It is seen that these $24$ patterns give rise to eight distinct equivalence classes of patterns based on the reversal and complementation operations.  The permutations avoiding the pattern in question in six of these classes are enumerated by Li \cite{Li}, and an additional case was found by Domagalski et al.~\cite{DLM}.  In three of these classes, it is shown that there are $C_{n-1}=\frac{1}{n}\binom{2n-2}{n-1}$ circular permutations of length $n$ that avoid the underlying pattern, while for two others, the permutations in question were shown to have cardinality given by $1+\sum_{i=0}^{n-2}i(i+1)^{n-i-2}$ for $n \geq 2$.  The enumeration of permutations in the remaining uncounted class corresponding to $\overline{23}41$ ($\equiv \overline{32}14$) was left as an open problem in \cite{Li}, which we address here.  Note that $\overline{23}41$ indeed corresponds to a trivial Wilf equivalence class and its enumerating sequence fails to occur in the OEIS \cite{Slo}.

In this paper, we find an explicit formula for the generating function (gf) that enumerates the circular permutations of length $n$ for $n \geq 1$ avoiding the vincular pattern $\overline{23}41$ (see Theorem \ref{mth1}  below).  In the next section, we find recurrences for three auxiliary arrays which when taken together refine the sequence of cardinalities that are sought.  In the third section, we rewrite these recurrences in terms of gf's which lead to functional equations satisfied by the gf's that can be solved explicitly to yield the desired formula.  We note that a more general gf result can be obtained wherein the class of permutations in question is enumerated according to the joint distribution of two parameters defined on the class.

\section{Recurrence formulas of arrays}

Let $\mathcal{A}_n$ denote the set of circular permutations of $[n]$ that avoid the vincular pattern $\overline{23}41$.  Consider fixing the position of the element $1$ within a circular permutation $\lambda$ and let $\lambda'$ denote the resulting linear permutation of length $n-1$ obtained by deleting $1$ from $\lambda$.  Then we have that $\lambda$ avoids $\overline{23}41$ if and only if $\lambda'$ avoids both $\overline{12}3$ and $41\overline{23}$.  To realize this, first let $\overline{xy}zw$ denote a potential occurrence of $\overline{23}41$ within $\lambda$.  Then consider whether or not $z$ occurs prior to the actual letter $1$ when starting from $y$ and proceeding clockwise.  If it does, then  $\overline{xy}z$ is an occurrence of $\overline{12}3$ in $\lambda'$.  If it does not, then $zw\overline{xy}$ must correspond to an occurrence of $41\overline{23}$ in $\lambda'$.  Conversely, if $\lambda'$ contains the patterns $\overline{12}3$ and $41\overline{23}$, then it is seen that $\lambda$ must contain $\overline{23}41$.

Let $\mathcal{L}_n$ denote the set of linear permutations of $[n]$ that avoid $\{\overline{12}3,41\overline{23}\}$ and we seek to enumerate the members of $\mathcal{L}_n$.  It will be convenient to hold out from consideration the permutation $(n-1)(n-2)\cdots 1n$. Denote by $\mathcal{L}_n^*$ the set $\mathcal{L}_n-\{(n-1)(n-2)\cdots 1n\}$ and we refine $\mathcal{L}_n^*$ as follows.  Let $\mathcal{B}_n$ for $n \geq 2$ be the subset of $\mathcal{L}_n^*$ consisting of those members in which $1$ occurs somewhere to the right of $n$ and let $\mathcal{C}_n$ for $n \geq 3$ be the subset of $\mathcal{L}_n^*$ in whose members $1$ occurs to the left of $n$ while $2$ occurs somewhere to its right.

Note that within a member of $\mathcal{L}_n$, all letters to the left of $n$ must occur in decreasing order so as to avoid $\overline{12}3$.  Suppose $n \geq 4$ and let $\rho \in \mathcal{L}_n^*-(\mathcal{B}_n\cup\mathcal{C}_n)$.  Then both $1$ and $2$ must occur to the left of $n$, implying $\rho$ can be decomposed as $\rho=\alpha\beta n \gamma$, where $\alpha$ is decreasing and possibly empty, $\beta$ consists of the elements of $[d]$ for some $d \geq 2$ in decreasing order with $d$ taken to be maximal and $\gamma$ is non-empty.  Note that $\rho \in \mathcal{L}_n^*$ implies $d \leq n-2$ and thus $d+1 \in \gamma$.  Observe further that the subsequence $d(d-1)\cdots 2$ within $\rho$ is superfluous concerning the avoidance of both $\overline{12}3$ and $41\overline{23}$, as it is decreasing and consists of an interval.  Indeed, only the $1$ would be needed in a potential occurrence of $41\overline{23}$ in which the role of the $1$ is played by a letter from $\beta$.  Further, no letter in $\beta$ can play the role of the $4$, $2$ or $3$ within an occurrence of the pattern $41\overline{23}$.

Thus, one may delete all elements of $[2,d]$ from $\rho$ (and subsequently subtract $d-1$ from each letter in $[d+1,n]$) to obtain a member of $\mathcal{C}_{n-d+1}$ for some $2 \leq d \leq n-2$.  Allowing $d$ to range over all possible values implies
\begin{equation}\label{andef0}
|\mathcal{L}_n^*|=|\mathcal{B}_n|+\sum_{d=1}^{n-2}|\mathcal{C}_{n-d+1}|, \qquad n \geq 2.
\end{equation}
In order to determine the cardinalities of $\mathcal{B}_n$ and $\mathcal{C}_n$, we refine these sets as follows.  Given $n \geq 2$ and $1 \leq i,j \leq n$ with $i \neq j$, let $\mathcal{B}_{n,i,j}$ and $\mathcal{C}_{n,i,j}$ denote the subsets of $\mathcal{B}_n$ and $\mathcal{C}_n$ consisting of those members that end in $i,j$.  For example, we have $\mathcal{B}_{5,3,2}=\{45132,51432,54132\}$ and $\mathcal{C}_{5,2,4}=\{15324,31524\}$. Let $b(n,i,j)=|\mathcal{B}_{n,i,j}|$ and $c(n,i,j)=|\mathcal{C}_{n,i,j}|$.  The initial values for $b(n,i,j)$ and $c(n,i,j)$ are as follows:  $b(2,2,1)=1$, with $b(2,1,2)=c(2,1,2)=c(2,2,1)=0$ for $n=2$; $b(3,1,3)=b(3,2,3)=b(3,3,2)=0$ and $b(3,1,2)=b(3,2,1)=b(3,3,1)=c(3,3,2)=1$, with $c(3,i,j)=0$ if $(i,j)\neq(3,2)$ for $n=3$.

Given $n \geq 2$ and $1 \leq j \leq n$, let $\mathcal{B}_{n,j}=\cup_{i=1,i\neq j}^n\mathcal{B}_{n,i,j}$ and $\mathcal{C}_{n,j}=\cup_{i=1,i\neq j}^n\mathcal{C}_{n,i,j}$, with $b(n,j)=|\mathcal{B}_{n,j}|$ and $c(n,j)=|\mathcal{C}_{n,j}|$. By the definitions, we have  $$b(n,j)=\sum_{i=1,i\neq j}^nb(n,i,j)$$ and
$$c(n,j)=\sum_{i=1,i\neq j}^nc(n,i,j).$$
The initial values for $b(n,j)$ and $c(n,j)$ are given by $b(2,1)=1$ and $b(2,2)=c(2,1)=c(2,2)=0$ for $n=2$, with $b(3,1)=2$, $b(3,2)=c(3,2)=1$ and $b(3,3)=c(3,1)=c(3,3)=0$ for $n=3$.  Note that $b(n,n)=c(n,n)=0$ for all $n \geq 2$ as the pertinent subsets of $\mathcal{B}_n$ and $\mathcal{C}_n$ are empty in both cases.

Let $a_n=|\mathcal{L}_n|$ for $n \geq 1$.  By \eqref{andef0} and taking into account the permutation $(n-1)(n-2)\cdots 1n$ which was excluded above, we have
\begin{equation}\label{andef}
a_n=1+\sum_{j=1}^{n}b(n,j)+\sum_{d=0}^{n-2}\sum_{i=1}^{n-d}c(n-d,i), \qquad n \geq 2,
\end{equation}
with $a_1=1$. Note that $|\mathcal{A}_n|=a_{n-1}$ for $n \geq 2$, with $a_1=1$, and hence we seek the generating function
$$A(x)=x+\sum_{n\geq 2}a_{n-1}x^n,$$
which counts all circular permutations of length $n$ that avoid $\overline{23}41$.

The arrays $b(n,i,j)$ and $c(n,i,j)$ may assume non-zero values only for $n \geq 2$ and $1\leq i,j \leq n$ with $i \neq j$ and are determined by the recurrences in the following lemmas.

\begin{lemma}\label{reclem1}
If $n \geq 3$, then
\begin{equation}\label{bneq1}
b(n,i,1)=b(n-1,i-1)+\sum_{d=2}^{i-1}c(n-i+d,d), \qquad 2 \leq i \leq n-1,
\end{equation}
\begin{equation}\label{bneq2}
b(n,i,j)=b(n-1,i-1), \qquad 2 \leq j<i \leq n-1,
\end{equation}
and
\begin{align}
b(n,1,j)&=2^{j-2}+\sum_{k=j+1}^{n-1}\sum_{d=2}^j\binom{j-2}{d-2}b(n-d,k-d)+\sum_{k=j+1}^{n-1}\sum_{d=2}^j\sum_{\ell=d}^{k-2}\binom{j-2}{d-2}c(n-\ell,k-\ell)\label{bneq3}
\end{align}
for $2 \leq j \leq n-1$, with $b(n,i,j)=0$ for $2 \leq i<j \leq n-1$ and $b(n,i,n)=0$, $b(n,n,j)=\delta_{j,1}$ if $n \geq 2$ and $i,j \in [n-1]$.
\end{lemma}
\begin{proof}
Note first that a member of $\mathcal{B}_n$ cannot end in $i,j$ such that $2 \leq i <j \leq n-1$ for otherwise $41\overline{23}$ would be present with $n1ij$, whence $b(n,i,j)=0$ for such $i$ and $j$.  Note that $\mathcal{B}_{n,n,j}$ is empty if $j \geq 2$ and consists of the single permutation $(n-1)(n-2)\cdots 2n1$ if $j=1$, whence $b(n,n,j)=\delta_{j,1}$.  Further, $\mathcal{B}_{n,i,n}$ is empty for all $i$ since $1$ must occur to the right of $n$.  For \eqref{bneq2}, note that the letter $j$ within members of $\mathcal{B}_{n,i,j}$ where $i,j \in [2,n-1]$ with $i>j$ is extraneous concerning the avoidance of both $\overline{12}3$ and $\overline{41}23$.  Thus, the $j$ may be deleted resulting in a member of $\mathcal{B}_{n-1,i-1}$, which implies \eqref{bneq2}.

To show \eqref{bneq1}, suppose $\lambda \in \mathcal{B}_{n,i,1}$ where $i \in [2,n-1]$.  Then $1$ may be deleted from $\lambda$ resulting in a member of $\mathcal{B}_{n-1,i-1}$ if $2$ also occurs to the right of $n$ within $\lambda$.  Otherwise, $2$ occurs to the left of $n$ and thus directly precedes $n$.  Consider the smallest $d \geq 2$ such that $d+1$ occurs somewhere to the right of $n$ within $\lambda$, noting $d \leq i-1$ since $i$ must occur to the right of $n$.  It is seen that each of the letters in $[3,d]$, along with $1$, may be deleted from $\lambda$, resulting in a member of $\mathcal{C}_{n-d+1,i-d+1}$.  Summing over all $d$ then gives $\sum_{d=2}^{i-1}c(n-d+1,i-d+1)=\sum_{d=2}^{i-1}c(n-i+d,d)$ possibilities for the case when $2$ occurs to the left of $n$ in $\lambda$, which implies \eqref{bneq1}.

Finally, to show \eqref{bneq3}, let $\rho \in \mathcal{B}_{n,1,j}$ where $j \in [2,n-1]$ and suppose $k$ is the rightmost letter within $\rho$ belonging to $[j+1,n]$. Then $\rho$ may be written as $\rho=\alpha\beta\gamma 1j$, where $\alpha$ is decreasing and $\beta$ is non-empty and starts with $n$ and ends in $k$ (with $k=n$ possible), whence $\gamma$ decreases as it has all of its letters in $[2,j-1]$. We consider cases based on $k$.  If $k=n$, then there are $2^{j-2}$ possibilities for the letters in $\gamma$, which determines $\rho$ completely in this case as all other letters in $[2,n-1]$ must then occur to the left of $n$ in decreasing order.  So assume $k<n$ and note $k>j$ implies that the section $\gamma 1j$ may be deleted from $\rho$.  Suppose $|\gamma|=d-2$ where $2 \leq d \leq j$ and let $\rho'$ denote the member of $\mathcal{L}_{n-d}$ that results from the deletion of $\gamma 1j$ from $\rho$ (followed by standardization).  If the element $1$ within $\rho'$ occurs to the right of $n-d$, then allowing  $k$ and $d$ to vary gives $\sum_{k=j+1}^{n-1}\sum_{d=2}^j\binom{j-2}{d-2}b(n-d,k-d)$ possibilities for $\rho$ in this case.

Otherwise, $1$ occurs (directly) to the left of $n-d$ and let $\ell\geq 1$ be maximal such that all members of $[\ell]$ occur to the left of $n-d$ within $\rho'$.  Note that $k$ lying to the right of $n$ in $\rho$ implies $\ell \leq k-d-1$.  Further, it is seen that all elements of $[2,\ell]$ may be deleted from $\rho'$ resulting in a member of $\mathcal{C}_{n-d-\ell+1,k-d-\ell+1}$.  Thus, considering all $k$, $d$ and $\ell$ yields
$\sum_{k=j+1}^{n-1}\sum_{d=2}^j\sum_{\ell=1}^{k-d-1}\binom{j-2}{d-2}c(n-\ell-d+1,k-\ell-d+1)$ possibilities for $\rho$ in this case.  This accounts for the second summation on the right side of \eqref{bneq3}, where $\ell$ has been replaced by $\ell-d+1$.  Combining with the prior cases implies \eqref{bneq3} and completes the proof.
\end{proof}

To give the recurrence for $c(n,i,j)$, we will need to consider the array $v(n,j)$ for $n \geq 1$ and $1 \leq j \leq n$ which enumerates the set of $\{\overline{12}3,1\overline{23}\}$-avoiding linear permutations of $[n]$ ending in $j$.

\begin{lemma}\label{reclem2}
If $n \geq 4$, then
\begin{equation}\label{cneq1}
c(n,i,2)=\sum_{d=2}^{i-1}c(n-i+d,d), \qquad 3 \leq i \leq n-1,
\end{equation}
\begin{equation}\label{cneq2}
c(n,i,j)=c(n-1,i-1), \qquad 3 \leq j <i \leq n-1,
\end{equation}
and
\begin{equation}\label{cneq3}
c(n,2,j)=\sum_{k=j+1}^{n-1}\sum_{d=3}^j\sum_{e=0}^{j-d}\binom{j-3}{d-3}\binom{j-d}{e}v(n-d-e-1,k-d-e), \qquad 3 \leq j \leq n-2,
\end{equation}
with $c(n,2,n-1)=2^{n-4}$ and $c(n,i,j)=0$ if $3 \leq i <j \leq n-1$.  Further, for $n \geq 2$, we have $c(n,i,n)=c(n,i,1)=c(n,1,j)=0$ for all $i$ and $j$, with $c(n,n,j)=\delta_{j,2}$ for $1 \leq j \leq n-1$.
\end{lemma}
\begin{proof}
Concerning the boundary values, note that $c(n,i,n)=c(n,i,1)=c(n,1,j)=0$ follows from the definitions and $c(n,n,j)=\delta_{j,2}$ since $\mathcal{C}_{n,n,j}$ is empty if $j \neq 2$, with $\mathcal{C}_{n,n,2}$ consisting of the single permutation $(n-1)(n-2)\cdots 1n2$.  Further, $c(n,i,j)=0$ for $3 \leq i <j \leq n-1$ since a member of $\mathcal{C}_n$ cannot end in an ascent $i,j$ where $i \geq 3$, for otherwise there would be an occurrence of $41\overline{23}$ witnessed by $n2ij$ as $2$ occurs to the right of $n$ by assumption.  Also, the final letter $j$ within a member of $\mathcal{C}_{n,i,j}$ where $3 \leq j <i \leq n-1$ may be deleted resulting in a member of $\mathcal{C}_{n-1,i-1}$, which implies \eqref{cneq2}.  To show \eqref{cneq1}, suppose $\lambda \in \mathcal{C}_{n,i,2}$ where $3 \leq i \leq n-1$ and consider the largest $d \geq 2$ such that all elements of $[3,d]$ occur to the left of $n$.  Then all elements of $[2,d]$ may be deleted from $\lambda$ resulting in a member of $\mathcal{C}_{n-d+1,i-d+1}$.  Summing over all $d$ then yields $\sum_{d=2}^{i-1}c(n-d+1,i-d+1)$ possibilities for $\lambda$ and replacing $d$ with $i+1-d$ in the sum implies \eqref{cneq1}.

To show \eqref{cneq3}, let $\rho \in \mathcal{C}_{n,2,j}$ where $3 \leq j \leq n-1$.  If $j=n-1$, then $\rho$ is expressible as $\rho=\alpha 1n\beta2(n-1)$, where $\alpha$ and $\beta$ are both decreasing (in order to avoid $\overline{12}3$).  Further, any $\rho$ of this form is seen to also avoid $41\overline{23}$.  As $\alpha$ may comprise any subset of $[3,n-2]$, the formula for $c(n,2,n-1)$ follows. So assume $n \geq 5$ and $3 \leq j \leq n-2$.  As in the proof of \eqref{bneq3} above, we let $k$ denote the rightmost letter within $\rho$ belonging to $[j+1,n]$.  Note that $k=n$ is not possible for $j \leq n-2$ would imply $n-1$ must occur to the left of $n$ and thus be the first letter of $\rho$.  But then $(n-1)12j$ is an occurrence $41\overline{23}$, so we must have $j+1 \leq k \leq n-1$.  Thus, $\rho$ can be decomposed as
$$\rho=\alpha 1n\beta \gamma 2j,$$
where $\alpha$ and $\gamma$ are possibly empty and $\beta$ is non-empty with last letter $k$ (and hence $\gamma$ contains only letters in $[3,j-1]$).

Then the following conditions on $\alpha$, $\beta$ and $\gamma$ are necessary for $\rho$ to avoid $\{\overline{12}3,41\overline{23}\}$: (i) $\alpha$ and $\gamma$ are decreasing, (ii) $\alpha$ contains only letters in $[3,j-1]$ and (iii) $\beta$ avoids the patterns $\overline{12}3$ and $1\overline{23}$.  The sections $\alpha$ and $\gamma$ must clearly both decrease so as to avoid $\overline{12}3$. Concerning (ii), note that if $x\geq j+1$ occurred in $\alpha$, then $x12j$ would correspond to a $41\overline{23}$ in $\rho$.  For (iii), the section $\beta$ must avoid $1\overline{23}$, for otherwise there would be a $41\overline{23}$ in $\rho$ where the role of $4$ is played by $n$. We claim that conditions (i)--(iii) are also sufficient for $\rho$ to avoid $\{\overline{12}3,41\overline{23}\}$.  Note that $\rho$ clearly avoids $\overline{12}3$ if (i)--(iii) are satisfied.  On the other hand, suppose to the contrary that $\rho$ contains an occurrence of $41\overline{23}$.  The only way in which this would be possible, given that $\rho$ satisfies (i)--(iii), is for there to exist $a>c>b$ with $a$ lying in $\alpha$ and $b,c$ in $\beta$ such that $b$ directly precedes $c$.  But then $c<a\leq j-1$ implies $c$ is not the last letter of $\beta$ and thus $bck$ would be an occurrence of $\overline{12}3$ in $\beta$, contrary to (iii), which establishes the claim.

To enumerate $\rho$ of the stated form satisfying (i)--(iii), let $|\alpha|=d-3$ and $|\gamma|=e$, where $3 \leq d \leq j$ and $0 \leq e \leq j-d$.  Then there are $\binom{j-3}{d-3} \binom{j-d}{e}$ ways in which to select the elements of $\alpha$ and $\gamma$, which must occur in decreasing order within their respective sections.  Upon standardizing, it is seen that $\beta$ is a permutation having last letter $k-(d-3)-e-3=k-d-e$ and length $n-d-e-1$, and hence is enumerated by $v(n-d-e-1,k-d-e)$.  Considering all possible values of $k$, $d$ and $e$ then yields the summation formula on the right side of \eqref{cneq3} and completes the proof.
\end{proof}

The array value $v(n,i)$ used above is itself defined recursively as follows.

\begin{lemma}\label{reclem3}
We have
\begin{equation}\label{vneq1}
v(n,1)=\sum_{i=1}^{n-1}v(n-1,i), \qquad n \geq 2,
\end{equation}
and
\begin{equation}\label{vneq2}
v(n,j)=\sum_{i=j}^{n-1}v(n-1,i)+\sum_{i=j+1}^n\sum_{d=2}^j\binom{j-2}{d-2}v(n-d,i-d), \qquad 2 \leq j \leq n-1,
\end{equation}
with $v(n,n)=1$ for all $n \geq 1$.
\end{lemma}
\begin{proof}
Let $\mathcal{V}_{n,i}$ denote the set of permutations enumerated by $v(n,i)$ and $\mathcal{V}_n=\cup_{i=1}^n\mathcal{V}_{n,i}$.  Note first that $v(n,n)=1$ since $\mathcal{V}_{n,n}$ consists of the single permutation $(n-1)(n-2)\cdots 1n$.  Formula \eqref{vneq1} follows from removal of $1$ from members of $\mathcal{V}_{n,1}$, which is seen to yield all members of $\mathcal{V}_{n-1}$.  To show \eqref{vneq2}, let $\lambda \in \mathcal{V}_{n,j}$ where $2 \leq j \leq n-1$.  If the penultimate letter of $\lambda$ is greater than $j$, then $j$ may be deleted yielding $\sum_{i=j}^{n-1}v(n-1,i)$ possibilities.  Otherwise, the penultimate letter of $\lambda$ must be $1$ in order to avoid $1\overline{23}$.  In this case, we may write $\lambda=\alpha i \beta 1 j$, where $i \in [j+1,n]$ and $\beta$ consists of letters in $[2,j-1]$ and is decreasing, with $\alpha$ or $\beta$ possibly empty.  Note that $i>j$ implies $\lambda$ avoids $\{\overline{12}3,1\overline{23}\}$ if and only if the initial section $\alpha i$ does.  Let $|\beta|=d-2$ for some $2 \leq d \leq j$.  Then there are $\binom{j-2}{d-2}$ ways in which to choose the letters in $\beta$ and hence $\binom{j-2}{d-2}v(n-d,i-d)$ possibilities for $\lambda$. Considering all values of $i$ and $d$ then accounts for the second sum on the right side of \eqref{vneq2} and completes the proof.
\end{proof}

\section{Computation of generating function formula}

In this section, we calculate an explicit formula for $A(x)$ by first determining the gf's of the arrays $v(n,i)$, $c(n,i,j)$ and $b(n,i,j)$.

\subsection{The generating function for $v(n,i)$}

Define $V_n(p)=\sum_{j=1}^nv(n,j)p^{j-1}$ and $V(x,p)=\sum_{n\geq1}V_n(p)x^n$. Note that $V(x,p)$ is the gf that enumerates the set of linear $\{\overline{12}3,1\overline{23}\}$-avoiding permutations of $[n]$ for $n \geq 1$ according to the last letter statistic (marked by $p$).  We seek to derive an explicit formula for $V(x,p)$.

In order to do so, first note that from \eqref{vneq2}, one gets
\begin{align*}
\sum_{n\geq3}\sum_{j=2}^{n-1}v(n,j)p^{j-1}x^n&=
\sum_{n\geq3}\sum_{j=2}^{n-1}\sum_{i=j}^{n-1}v(n-1,i)p^{j-1}x^n\\
&\quad+\sum_{n\geq3}\sum_{j=2}^{n-1}\sum_{i=j+1}^n\sum_{d=2}^j\binom{j-2}{d-2}v(n-d,i-d)p^{j-1}x^n,
\end{align*}
which implies by \eqref{vneq1},
\begin{align*}
&V(x,p)-xV(x,1)-\frac{x}{1-px}-\frac{px}{1-p}(V(x,1)-V(x,p))\\
&=\sum_{n\geq3}\sum_{i=3}^{n}\sum_{d=2}^{i-1}\sum_{j=d}^{i-1}\binom{j-2}{d-2}v(n-d,i-d)p^{j-1}x^n\\
&=\sum_{n\geq 3}\sum_{d=0}^{n-3}\sum_{j=0}^{n-d-3}\sum_{i=j+1}^n\binom{j+d}{d}v(n-d-2,i-d-2)p^{j+d+1}x^n\\
&=\sum_{n\geq0}\sum_{d\geq0}\sum_{j\geq0}\sum_{i=j+d+3}^{n+j+d+3}\binom{j+d}{d}v(n+j+1,i-d-2)p^{j+d+1}x^{n+j+d+3}\\
&=\sum_{n\geq0}\sum_{d\geq0}\sum_{j\geq0}\sum_{i=0}^{n}\binom{j+d}{d}v(n+j+1,i+j+1)p^{j+d+1}x^{n+j+d+3}\\
&=\sum_{j\geq0}\sum_{n\geq j}\sum_{i=j}^nv(n+1,i+1)\frac{p^{j+1}x^{n+3}}{(1-px)^{j+1}}=\frac{px^2}{1-px}\sum_{n\geq1}\sum_{i=1}^n\sum_{j=0}^{i-1}v(n,i)\frac{p^jx^n}{(1-px)^j}\\
&=\frac{px^2}{1-p-px}\sum_{n\geq1}\sum_{i=1}^nv(n,i)x^n\left(1-\frac{p^i}{(1-px)^i}\right)\\
&=\frac{px^2}{1-p-px}\left(V(x,1)-\frac{p}{1-px}V(x,\frac{p}{1-px})\right).
\end{align*}

Hence, the gf $V(x,p)$ satisfies the functional equation
\begin{align}\label{eqFV}
V(x,p)&=\frac{(1-p)x}{(1-p+px)(1-px)}+\frac{(1-p-p^2x)x}{(1-p+px)(1-p-px)}V(x,1)\nonumber\\
&\quad-\frac{(1-p)p^2x^2}{(1-px)(1-p+px)(1-p-px)}V(x,\frac{p}{1-px}).
\end{align}
Note that $V(x,0)=x+xV(x,1)$, upon taking $p=0$ in \eqref{eqFV}, which may also be realized combinatorially from the definition of $V(x,p)$. Thus,
\begin{align}\label{eqFV1}
V(x,p)&=\frac{(-1+p+px-p^2x^2)px}{(1-px)(1-p+px)(1-p-px)}+\frac{1-p-p^2x}{(1-p+px)(1-p-px)}V(x,0)\nonumber\\
&\quad-\frac{(1-p)p^2x^2}{(1-px)(1-p+px)(1-p-px)}V(x,\frac{p}{1-px}).
\end{align}

Let $a(p)=\frac{(-1+p+px-p^2x^2)px}{(1-px)(1-p+px)(1-p-px)}+\frac{1-p-p^2x}{(1-p+px)(1-p-px)}V(x,0)$ and $b(p)=\frac{(1-p)p^2x^2}{(1-px)(1-p+px)(1-p-px)}$.
By iterating \eqref{eqFV1} an infinite number of times (assuming $|x|<1$ is sufficiently small in absolute in value), we obtain
\begin{align*}
V(x,p)&=\sum_{j\geq0}(-1)^ja\left(\frac{p}{1-jpx}\right)\prod_{i=0}^{j-1}b\left(\frac{p}{1-ipx}\right)\\
&=\sum_{j\geq0}(-1)^j\frac{(-1+p-jp^2x+(2j+1)px-(j^2+j+1)p^2x^2)p^{2j+1}x^{2j+1}}
{(1-p+px)\prod_{i=1}^{j+1}(1-ipx)\prod_{i=1}^{j+1}(1-p-ipx)}\\
&\quad+V(x,0)\sum_{j\geq0}(-1)^j\frac{((1-jpx)^2-p+(j-1)p^2x)p^{2j}x^{2j}}{
(1-p+px)\prod_{i=1}^j(1-ipx)\prod_{i=1}^{j+1}(1-p-ipx)}.
\end{align*}
Thus,
\begin{align*}
V(x,1)&=-\sum_{j\geq0}\frac{(j+1-(j^2+j+1)x)x^j}
{(j+1)!\prod_{i=1}^{j+1}(1-ix)}-V(x,0)\sum_{j\geq0}\frac{((1-jx)^2-1+(j-1)x)x^{j-2}}{
(j+1)!\prod_{i=1}^j(1-ix)},
\end{align*}
which, by  $V(x,0)=x+xV(x,1)$, implies
\begin{align*}
V(x,0)&=\frac{x-\sum_{j\geq0}\frac{(j+1-(j^2+j+1)x)x^{j+1}}
{(j+1)!\prod_{i=1}^{j+1}(1-ix)}}{1+\sum_{j\geq0}\frac{((1-jx)^2-1+(j-1)x)x^{j-1}}{
(j+1)!\prod_{i=1}^j(1-ix)}}.
\end{align*}

Hence, we can state the following result.

\begin{lemma}\label{lem1}
The generating function $V(x,p)$, is given by
\begin{align*}
V(x,p)&=\sum_{j\geq0}(-1)^j\frac{(-1+p-jp^2x+(2j+1)px-(j^2+j+1)p^2x^2)p^{2j+1}x^{2j+1}}
{(1-p+px)\prod_{i=1}^{j+1}(1-ipx)\prod_{i=1}^{j+1}(1-p-ipx)}\\
&\quad+V(x,0)\sum_{j\geq0}(-1)^j\frac{((1-jpx)^2-p+(j-1)p^2x)p^{2j}x^{2j}}{
(1-p+px)\prod_{i=1}^j(1-ipx)\prod_{i=1}^{j+1}(1-p-ipx)},
\end{align*}
where $$V(x,0)=\frac{\sum_{j\geq1}\frac{(j+1-(j^2+j+1)x)x^{j+1}}
{(j+1)!\prod_{i=1}^{j+1}(1-ix)}}{\sum_{j\geq1}\frac{(j+1-j^2x)x^j}{
(j+1)!\prod_{i=1}^j(1-ix)}}.$$
\end{lemma}

\subsection{The generating function for $c(n,i,j)$}

Define $C_n(v,u)=\sum_{i=1}^n\sum_{j=1}^nc(n,i,j)v^{i-2}u^{j-2}$ and $C(x,v,u)=\sum_{n\geq2}C_n(v,u)x^n$.  Then the gf $C(x,v,u)$ enumerates the set of linear $\{\overline{12}3,41\overline{23}\}$-avoiding permutations of $[n]$ for $n \geq 3$ such that $1$ occurs to the left of $n$ and $2$ to the right according to the penultimate and final letter statistics (marked by $v$ and $u$, respectively). A similar interpretation may be given for $B(x,v,u)$ defined below.  In this subsection, we seek to determine a formula for $C(x,v,u)$.

In order to do so, first note that from the definitions and the basic properties of $c(n,i,j)$, we have
\begin{align}\label{eqcca1}
C(x,v,u)&=\sum_{n\geq4}\sum_{j=3}^{n-1}c(n,2,j)u^{j-2}x^n
+\sum_{n\geq3}\sum_{i=3}^n\sum_{j=2}^{n-1}c(n,i,j)v^{i-2}u^{j-2}x^n.
\end{align}
Note that by \eqref{cneq3}, we have
\begin{align*}
&\sum_{n\geq4}\sum_{j=3}^{n-1}c(n,2,j)u^{j-2}x^n-\sum_{n\geq4}2^{n-4}u^{n-3}x^n\\
&=\sum_{n\geq5}\sum_{j=3}^{n-2}c(n,2,j)u^{j-2}x^n\\
&=\sum_{n\geq5}\sum_{j=3}^{n-2}\sum_{k=j+1}^{n-1}\sum_{d=3}^j\sum_{e=0}^{j-d}\binom{j-3}{d-3}\binom{j-d}{e}v(n-d-e-1,k-d-e)u^{j-2}x^n\\
&=\sum_{j\geq3}\sum_{d=3}^{j}\sum_{e=0}^{j-d}\sum_{k \geq j+1}\sum_{n\geq k+1}\binom{j-3}{d-3}\binom{j-d}{e}v(n-d-e-1,k-d-e)u^{j-2}x^n\\
&=\sum_{j\geq3}\sum_{d=3}^{j}\sum_{e=d}^{j}\sum_{k \geq j+1}\sum_{n\geq k+1}\binom{j-3}{d-3}\binom{j-d}{e-d}v(n-e-1,k-e)u^{j-2}x^n\\
&=\sum_{j\geq3}\sum_{d=3}^{j}\sum_{e=d}^{j}\sum_{k \geq j-e+1}\sum_{n\geq k}\binom{j-3}{d-3}\binom{j-d}{e-d}v(n,k)u^{j-2}x^{n+e+1}\\
&=\sum_{d\geq3}\sum_{e \geq d}\sum_{j \geq e}\sum_{k \geq j-e+1}\sum_{n\geq k}\binom{j-3}{d-3}\binom{j-d}{e-d}v(n,k)u^{j-2}x^{n+e+1}\\
&=\sum_{d\geq3}\sum_{e\geq0}\sum_{j\geq0}\sum_{k\geq j+1}\sum_{n \geq k}\binom{j+e+d-3}{d-3}\binom{j+e}{e}v(n,k)u^{j+e+d-2}x^{n+e+d+1}\\
&=\sum_{j \geq0}\sum_{n\geq j+1}\sum_{k=j+1}^nv(n,k)ux^{n-j+4}\sum_{e \geq0}\binom{j+e}{e}\sum_{d \geq 3}\binom{j+e+d-3}{d-3}(ux)^{j+e+d-3}\\
&=\sum_{j\geq0}\sum_{n\geq j+1}\sum_{k=j+1}^nv(n,k)\frac{u^{j+1}x^{n+4}}{(1-2ux)^{j+1}}=\sum_{n\geq1}\sum_{k=1}^n\sum_{j=0}^{k-1}v(n,k)\frac{u^{j+1}x^{n+4}}{(1-2ux)^{j+1}}\\
&=\frac{ux^4}{1-u-2ux}\sum_{n\geq1}\sum_{k=1}^nv(n,k)x^n\left(1-\frac{u^{k}}{(1-2ux)^k}\right)\\
&=\frac{ux^4}{1-u-2ux}\left(V(x,1)-\frac{u}{1-2ux}V(x,\frac{u}{1-2ux})\right).
\end{align*}

Hence, by \eqref{eqcca1}, we get
\begin{align}\label{eqcca2}
C(x,v,u)&=\frac{ux^4}{1-u-2ux}\left(V(x,1)-\frac{u}{1-2ux}V(x,\frac{u}{1-2ux})\right)+\frac{ux^4}{1-2ux}\nonumber\\
&\quad+\sum_{n\geq3}\sum_{i=3}^n\sum_{j=2}^{n-1}c(n,i,j)v^{i-2}u^{j-2}x^n.
\end{align}
By \eqref{cneq1} and \eqref{cneq2}, we have
\begin{align*}
&\sum_{n\geq3}\sum_{i=3}^n\sum_{j=2}^{n-1}c(n,i,j)v^{i-2}u^{j-2}x^n\\
&=\sum_{n\geq5}\sum_{i=3}^{n-2}\sum_{j=i+1}^{n-1}c(n,i,j)v^{i-2}u^{j-2}x^n
+\sum_{n\geq3}\sum_{i=3}^n\sum_{j=2}^{i-1}c(n,i,j)v^{i-2}u^{j-2}x^n\\
&=\sum_{n\geq5}\sum_{i=4}^{n-1}\sum_{j=3}^{i-1}c(n,i,j)v^{i-2}u^{j-2}x^n
+\sum_{n\geq3}\sum_{i=3}^nc(n,i,2)v^{i-2}x^n\\
&=u\sum_{n\geq5}\sum_{i=4}^{n-1}\sum_{\ell=1}^{n-1}c(n-1,\ell,i-1)\frac{v^{i-2}(1-u^{i-3})x^n}{1-u}
+\sum_{n\geq3}\sum_{i=3}^nc(n,i,2)v^{i-2}x^n\\
&=\frac{uvx}{1-u}(C(x,1,v)-C(x,1,uv))
+\sum_{n\geq4}\sum_{i=3}^{n-1}c(n,i,2)v^{i-2}x^n
+\sum_{n\geq3}c(n,n,2)v^{n-2}x^n\\
&=\frac{uvx}{1-u}(C(x,1,v)-C(x,1,uv))
+\sum_{n\geq4}\sum_{i=3}^{n-1}\sum_{d=2}^{i-1}c(n-i+d,d)v^{i-2}x^n
+\frac{vx^3}{1-vx}\\
&=\frac{uvx}{1-u}(C(x,1,v)-C(x,1,uv))
+\frac{vx}{1-vx}\sum_{n\geq3}\sum_{d=2}^{n-1}c(n,d)v^{d-2}x^n
+\frac{vx^3}{1-vx}\\
&=\frac{uvx}{1-u}(C(x,1,v)-C(x,1,uv))
+\frac{vx}{1-vx}C(x,1,v)+\frac{vx^3}{1-vx}.
\end{align*}

By \eqref{eqcca2}, it follows that $C(x,v,u)$ satisfies
\begin{align}\label{eqcca3}
C(x,v,u)&=\frac{ux^4}{1-u-2ux}\left(V(x,1)-\frac{u}{1-2ux}V(x,\frac{u}{1-2ux})\right)+\frac{ux^4}{1-2ux}\nonumber\\
&\quad+\frac{uvx}{1-u}(C(x,1,v)-C(x,1,uv))
+\frac{vx}{1-vx}C(x,1,v)+\frac{vx^3}{1-vx}.
\end{align}
To solve \eqref{eqcca3}, we apply the \emph{kernel method} \cite{HouM} and take $u=1/(1-x)$ and $v=1$ to obtain
\begin{align}\label{eqccx11}
C(x,1,1)&=\frac{((1-x)V(x,\frac{1}{1-3x})-(1-4x+3x^2)V(x,1)+3(1-2x-x^2))x^3}{3(1-2x)(1-3x)}.
\end{align}
By \eqref{eqcca3} with $v=1$, we have
\begin{align}\label{eqccx1u}
C(x,1,u)&=\frac{(1-ux)x}{(1-x)(1-u+ux)}C(x,1,1)
+\frac{(1-u)ux^4}{(1-u+ux)(1-u-2ux)}V(x,1)\nonumber\\
&\quad-\frac{(1-u)u^2x^4}{(1-u+ux)(1-u-2ux)(1-2ux)}V(x,\frac{u}{1-2ux})\nonumber\\
&\quad+\frac{(1-u)(1-ux-ux^2)x^3}{(1-x)(1-u+ux)(1-2ux)}.
\end{align}

Hence, we can state the following result.

\begin{lemma}\label{lem2}
The generating function $C(x,v,u)$ is given by
\begin{align*}
C(x,v,u)&=\frac{ux^4}{1-u-2ux}\left(V(x,1)-\frac{u}{1-2ux}V(x,\frac{u}{1-2ux})\right)+\frac{ux^4}{1-2ux}\\
&\quad+\frac{uvx}{1-u}(C(x,1,v)-C(x,1,uv))
+\frac{vx}{1-vx}C(x,1,v)+\frac{vx^3}{1-vx},
\end{align*}
where $V(x,p)$ is given in Lemma \ref{lem1} and $C(x,1,u)$ by \eqref{eqccx1u}.
\end{lemma}

\subsection{The generating function for $b(n,i,j)$}

Define $B_n(v,u)=\sum_{i=1}^n\sum_{j=1}^nb(n,i,j)v^{i-1}u^{j-1}$ and $B(x,v,u)=\sum_{n\geq2}B_n(v,u)x^n$. From the definitions and basic conditions on $b(n,i,j)$, we have
\begin{align}\label{eqbba1}
B(x,v,u)&=\sum_{n\geq4}\sum_{i=2}^{n-2}\sum_{j=i+1}^{n-1}b(n,i,j)v^{i-1}u^{j-1}x^n
+\sum_{n\geq3}\sum_{j=2}^{n-1}b(n,1,j)u^{j-1}x^n\nonumber\\
&\quad+\sum_{n\geq2}\sum_{i=2}^{n}\sum_{j=1}^{i-1}b(n,i,j)v^{i-1}u^{j-1}x^n\nonumber\\
&=\sum_{n\geq3}\sum_{j=2}^{n-1}b(n,1,j)u^{j-1}x^n
+\sum_{n\geq4}\sum_{i=3}^{n-1}\sum_{j=2}^{i-1}b(n,i,j)v^{i-1}u^{j-1}x^n
+\sum_{n\geq2}\sum_{i=2}^nb(n,i,1)v^{i-1}x^n.
\end{align}

By \eqref{bneq2}, we get
\begin{align*}
&\sum_{n\geq4}\sum_{i=3}^{n-1}\sum_{j=2}^{i-1}b(n,i,j)v^{i-1}u^{j-1}x^n\\
&=\sum_{n\geq4}\sum_{i=3}^{n-1}\sum_{j=2}^{i-1}b(n-1,i-1)v^{i-1}u^{j-1}x^n=\frac{uvx}{1-u}(B(x,1,v)-B(x,1,uv)).
\end{align*}
Hence, by \eqref{eqbba1}, we have
\begin{align}\label{eqbba2}
B(x,v,u)&=\sum_{n\geq3}\sum_{j=2}^{n-1}b(n,1,j)u^{j-1}x^n
+\frac{uvx}{1-u}(B(x,1,v)-B(x,1,uv))\nonumber\\
&\quad+\sum_{n\geq2}\sum_{i=2}^nb(n,i,1)v^{i-1}x^n.
\end{align}
Applying \eqref{bneq1} yields
\begin{align*}
&\sum_{n\geq2}\sum_{i=2}^nb(n,i,1)v^{i-1}x^n\\
&=\sum_{n\geq3}\sum_{i=2}^{n-1}\left(b(n-1,i-1)+\sum_{d=2}^{i-1}c(n-i+d,d)\right)v^{i-1}x^n
+\frac{vx^2}{1-vx}\\
&=x\sum_{n\geq2}\sum_{i=1}^{n-1}b(n,i)v^ix^n
+\sum_{n\geq4}\sum_{i=3}^{n-1}\sum_{d=2}^{i-1}c(n-i+d,d)v^{i-1}x^n
+\frac{vx^2}{1-vx}\\
&=vxB(x,1,v)+\sum_{d\geq2}\sum_{i\geq1}\sum_{n\geq d+1}c(n,d)v^{i+d-1}x^{n+i}+\frac{vx^2}{1-vx}\\
&=vxB(x,1,v)+\frac{v^2x}{1-vx}C(x,1,v)+\frac{vx^2}{1-vx}.
\end{align*}

Hence, by \eqref{eqbba2}, we get
\begin{align}\label{eqbba3}
B(x,v,u)&=\sum_{n\geq3}\sum_{j=2}^{n-1}b(n,1,j)u^{j-1}x^n
+\frac{vx}{1-u}(B(x,1,v)-uB(x,1,uv))\nonumber\\
&\quad+\frac{v^2x}{1-vx}C(x,1,v)+\frac{vx^2}{1-vx}.
\end{align}
By \eqref{bneq3}, we obtain
\begin{align}
&\sum_{n\geq3}\sum_{j=2}^{n-1}b(n,1,j)u^{j-1}x^n\notag\\
&=\sum_{n\geq3}\sum_{j=2}^{n-1}2^{j-2}u^{j-1}x^n+
\sum_{n\geq4}\sum_{j=2}^{n-2}\sum_{k=j+1}^{n-1}\sum_{d=2}^j\binom{j-2}{d-2}b(n-d,k-d)u^{j-1}x^n\notag\\
&\quad+\sum_{n\geq4}\sum_{j=2}^{n-2}\sum_{k=j+1}^{n-1}\sum_{d=2}^j\sum_{\ell=d}^{k-2}\binom{j-2}{d-2}c(n-\ell,k-\ell)u^{j-1}x^n\notag\\
&=\frac{ux^3}{(1-x)(1-2ux)}+
\sum_{d\geq2}\sum_{j\geq d}\sum_{n\geq j+2}\sum_{k=j+1}^{n-1}\binom{j-2}{d-2}b(n-d,k-d)u^{j-1}x^n\notag\\
&\quad+\sum_{d\geq2}\sum_{j\geq d}\sum_{k\geq j+1}\sum_{n \geq k+1}\sum_{\ell=d}^{k-2}\binom{j-2}{d-2}c(n-\ell,k-\ell)u^{j-1}x^n.\label{eqbba3a}
\end{align}

We simplify the two multi-sum expressions appearing in \eqref{eqbba3a} as follows.  Replacing $j$, $n$ and $k$ with $j+d$, $n+d$ and $k+d$, respectively, in the first yields
\begin{align*}
&\sum_{d\geq2}\sum_{j\geq d}\sum_{n\geq j+2}\sum_{k=j+1}^{n-1}\binom{j-2}{d-2}b(n-d,k-d)u^{j-1}x^n\\
&=\sum_{d\geq2}\sum_{j\geq 0}\sum_{n\geq j+2}\sum_{k=j+1}^{n-1}\binom{j+d-2}{d-2}b(n,k)u^{j+d-1}x^{n+d}\\
&=\sum_{j\geq 0}\sum_{n\geq j+2}\sum_{k=j+1}^{n-1}b(n,k)\frac{u^{j+1}x^{n+2}}{(1-ux)^{j+1}}=\frac{ux^2}{1-u-ux}\sum_{n\geq 2}\sum_{k=1}^{n-1}b(n,k)x^n\left(1-\left(\frac{u}{1-ux}\right)^k\right)\\
&=\frac{ux^2}{1-u-ux}\left(B(x,1,1)-\frac{u}{1-ux}B(x,1,\frac{u}{1-ux})\right).
\end{align*}

For the second multi-sum expression in \eqref{eqbba3a}, first note that we may define $c(n,k)$ to be zero for all $n$ and $k$ where it is not the case that $n>k\geq 2$.  Then the innermost sum where $d \leq \ell \leq k-2$ may be replaced with $\ell \geq d$, which allows for this sum to be readily interchanged with the others.  This results in
\begin{align*}
&\sum_{d\geq2}\sum_{\ell\geq d}\sum_{j \geq d}\sum_{k\geq j+1}\sum_{n \geq k+1}\binom{j-2}{d-2}c(n-\ell,k-\ell)u^{j-1}x^n\\
&=\sum_{d\geq2}\sum_{\ell\geq 0}\sum_{j \geq 0}\sum_{k\geq j+1}\sum_{n \geq k+1}\binom{j+d-2}{d-2}c(n-\ell,k-\ell)u^{j+d-1}x^{n+d}\\
&=\sum_{d\geq2}\sum_{\ell\geq 0}\sum_{k\geq \ell+2}\sum_{j=0}^{k-1}\sum_{n \geq k+1}\binom{j+d-2}{d-2}c(n-\ell,k-\ell)u^{j+d-1}x^{n+d},
\end{align*}
where we have interchanged the third and fourth sums and replaced $k \geq 1$ with $k \geq \ell +2$ since $c(n-\ell,k-\ell)$ may assume non-zero values only when $k \geq \ell+2$.

Replacing $k$ and $n$ with $k+\ell$ and $n+\ell$ in the last expression then gives
\begin{align*}
&\sum_{d\geq2}\sum_{\ell\geq 0}\sum_{k\geq 2}\sum_{j=0}^{k+\ell-1}\sum_{n \geq k+1}\binom{j+d-2}{d-2}c(n,k)u^{j+d-1}x^{n+d+\ell}\\
&=\sum_{\ell\geq 0}\sum_{k\geq 2}\sum_{j=0}^{k+\ell-1}\sum_{n \geq k+1}c(n,k)\frac{u^{j+1}x^{n+\ell+2}}{(1-ux)^{j+1}}\\
&=\frac{ux^2}{1-u-ux}\sum_{\ell\geq 0}\sum_{k\geq 2}\sum_{n \geq k+1}c(n,k)x^{n+\ell}\left(1-\left(\frac{u}{1-ux}\right)^{k+\ell}\right)\\
&=\frac{ux^2}{1-u-ux}\sum_{n\geq 3}\sum_{k=2}^{n-1}c(n,k)\left(\frac{x^n}{1-x}-\left(\frac{u}{1-ux}\right)^k\cdot\frac{x^n}{1-\frac{ux}{1-ux}}\right)\\
&=\frac{ux^2}{1-u-ux}\left(\frac{1}{1-x}C(x,1,1)
-\frac{u^2}{(1-ux)(1-2ux)}C(x,1,\frac{u}{1-ux})\right).
\end{align*}

Combining the prior results, we have by \eqref{eqbba3a},
\begin{align*}
&\sum_{n\geq3}\sum_{j=2}^{n-1}b(n,1,j)u^{j-1}x^n\\
&=\frac{ux^3}{(1-x)(1-2ux)}+\frac{ux^2}{1-u-ux}\left(B(x,1,1)-\frac{u}{1-ux}B(x,1,\frac{u}{1-ux})\right)\\
&\quad+\frac{ux^2}{1-u-ux}\left(\frac{1}{1-x}C(x,1,1)
-\frac{u^2}{(1-ux)(1-2ux)}C(x,1,\frac{u}{1-ux})\right).
\end{align*}
Hence, by \eqref{eqbba3}, we have that $B(x,v,u)$ satisfies
\begin{align}\label{eqbba4}
B(x,v,u)&=\frac{ux^3}{(1-x)(1-2ux)}+\frac{ux^2}{1-u-ux}\left(B(x,1,1)-\frac{u}{1-ux}B(x,1,\frac{u}{1-ux})\right)\nonumber\\
&\quad+\frac{ux^2}{1-u-ux}\left(\frac{1}{1-x}C(x,1,1)
-\frac{u^2}{(1-ux)(1-2ux)}C(x,1,\frac{u}{1-ux})\right)\nonumber\\
&\quad+\frac{vx}{1-u}(B(x,1,v)-uB(x,1,uv))+\frac{v^2x}{1-vx}C(x,1,v)
+\frac{vx^2}{1-vx}.
\end{align}
By substituting $v=1$ into \eqref{eqbba4}, we obtain
\begin{align*}
B(x,1,u)&=\frac{(u^2x+u-1)x}{(ux-u+1)(ux+u-1)}B(x,1,1)
-\frac{(1-u)^2x}{(ux-u+1)(ux+u-1)(1-x)}C(x,1,1)\\
&\quad+\frac{(1-u)u^3x^2}{(ux-u+1)(ux+u-1)(ux-1)(2ux-1)}C(x,1,\frac{u}{1-ux})\\
&\quad+\frac{(1-u)(1-ux)x^2}{(ux-u+1)(1-x)(1-2ux)}\\
&\quad-\frac{(1-u)u^2x^2}{(ux-u+1)(ux+u-1)(ux-1)}B(x,1,\frac{u}{1-ux}).
\end{align*}

Note that it is not possible to find $B(x,1,1)$ directly from the prior equation by taking $u=1$.  However, we may iterate the preceding equation to first obtain the following explicit formula for $B(x,1,u)$.

\begin{lemma}\label{lem3}
The generating function $B(x,1,u)$ is given by
\begin{align*}
B(x,1,u)&=B(x,1,1)\sum_{j\geq0}\frac{(-1)^j((1-jux)^2+(j-1)u^2x-u)u^{2j}x^{2j+1}}
{(1-u+ux)\prod_{i=1}^j(1-iux)\prod_{i=1}^{j+1}(1-u-iux)} \\ &\quad+C(x,1,1)\sum_{j\geq0}\frac{(-1)^j(1-u-jux)^2u^{2j}x^{2j+1}}{(1-x)(1-u+ux)
\prod_{i=1}^j(1-iux)\prod_{i=1}^{j+1}(1-u-iux)} \\ &\quad-\sum_{j\geq0}\frac{(-1)^j(1-u-jux)u^{2j+3}x^{2j+2}}{(1-u+ux)
\prod_{i=1}^{j+2}(1-iux)\prod_{i=1}^{j+1}(1-u-iux)}C\left(x,1,\frac{u}{1-(j+1)ux}\right)\\
&\quad+\sum_{j\geq0}\frac{(-1)^j(1-(j+1)ux)^2(1-u-jux)u^{2j}x^{2j+2}}{(1-x)
(1-u+ux)\prod_{i=1}^{j+2}(1-iux)\prod_{i=1}^{j}(1-u-iux)}.
\end{align*}
\end{lemma}

Now one may take $u=1$ in Lemma \ref{lem3}, and solve, to obtain $B(x,1,1)$ explicitly.

\begin{lemma}\label{lem4}
The generating function $B(x,1,1)$ is given by
\small $$-\frac{\frac{C(x,1,1)}{1-x}\sum_{j\geq1}\frac{j^2x^{j+1}}{(j+1)!\prod_{i=1}^j(1-ix)}
+\sum_{j\geq1}\frac{jx^{j+1}}{(j+1)!\prod_{i=1}^{j+2}(1-ix)}C\left(x,1,\frac{1}{1-(j+1)x}\right) +\frac{1}{1-x}\sum_{j\geq1}\frac{(1-(j+1)x)^2x^{j+2}}{(j-1)!\prod_{i=1}^{j+2}(1-ix)}}
{\sum_{j\geq1}\frac{(j^2x-j-1)x^j}{(j+1)!\prod_{i=1}^j(1-ix)}}.$$ \normalsize
\end{lemma}

\subsection{Explicit formula for $A(x)$}

By \eqref{andef}, the desired gf formula $A(x)=x+\sum_{n\geq 2}a_{n-1}x^n$ can be stated explicitly as follows.

\begin{theorem}\label{mth1}
The generating function for the number of circular permutations of length $n$ for $n \geq 1$ that avoid the vincular pattern $\overline{23}41$ is given by
$$A(x)=\frac{x}{1-x}+xB(x,1,1)+\frac{x}{1-x}C(x,1,1),$$
where $B(x,1,1)$ is given in Lemma \ref{lem4} and $C(x,1,1)$ by \eqref{eqccx11}.
\end{theorem}

For instance, from the formula for $A(x)$, we have that $a_n=|\mathcal{A}_{n+1}|$ for $1 \leq n \leq 30$ is given by
\begin{center}
\begin{tabular}{ll||ll||ll}
$n$&$a_n$&$n$&$a_n$&$n$&$a_n$\\\hline\hline
1&1&2&2&3&5\\
4&15&5 &50&6&180\\
7&690&8&2792&9&11857\\
10&52633&11&243455&12&1170525\\
13&5837934&14&30151474&15&161021581\\
16&888001485&17&5051014786&18&29600662480\\
19&178541105770&20&1107321666920&21&7055339825171\\
22&46142654894331&23&309513540865544&24&2127744119042216\\ 25&14979904453920111&26&107932371558460341&27&795363217306369817\\ 28&5990768203554158167&29&46094392105916344968&30&362092868720288824992.
\end{tabular}
\end{center}

Based on the terms of the sequence $a_n$ for $1 \leq n \leq 130$, we conjecture that there exists no constant $c$ such that $a_n<c^n$ for all $n\geq1$.  Further, we conjecture that the inequality $a_n^{n+1}<a_{n+1}^n$ holds for all $n \geq 1$.

We conclude by noting that a more general gf result in terms of two parameters defined on $\mathcal{A}_n$ can be obtained from the preceding as follows.  First note that the penultimate and final letter statistics on $\mathcal{L}_{n-1}$ correspond to (one less than) the values of the two letters directly prior to $1$ (when proceeding clockwise) within a member of $\mathcal{A}_n$ and are thus marked by $v$ and $u$.  Let $i+2$ and $j+2$ for some $i,j \geq 0$ denote the two letters directly preceding $1$ within a member of $\mathcal{A}_n$.  Let  $A(x,v,u)$ be the gf enumerating the members of $\mathcal{A}_n$ according to the values of $i$ and $j$ (marked by $v$ and $u$, respectively).  Then it is seen that $A(x,v,u)$ is given by
\begin{equation}\label{gener}
A(x,v,u)=\frac{x+(1-u)x^2}{1-ux}+xB(x,v,u)+\frac{uvx}{1-uvx}C(x,v,u),
\end{equation}
where $C(x,v,u)$ is as in Lemma \ref{lem2} and $B(x,v,u)$ is given by \eqref{eqbba4}, together with Lemmas \ref{lem2}--\ref{lem4}.

\end{document}